\documentclass[10pt, reqno]{amsart}
\usepackage{graphicx, amssymb, amsmath, amsthm}
\numberwithin{equation}{section}

\newcommand{\R}{\mathbb{R}}
\newcommand{\C}{\mathbb{C}}
\newcommand{\wh}[1]{\widehat{#1}}

\newcommand{\norm}[1]{\| #1\|}

\newcommand{\eps}{\varepsilon}

\newcommand{\xonorm}[3]{\|#1\|_{L_t^{#2}L_x^{#3}}}
\newcommand{\xnorm}[4]{\|#1\|_{L_t^{#2}L_x^{#3}(#4\times\R^d)}}
\newcommand{\xonorms}[2]{\|#1\|_{L_{t,x}^{#2}}}

\newcommand{\xnorms}[3]{\|#1\|_{L_{t,x}^{#2}(#3\times\R^d)}}

\newcommand{\nsc}{\vert\nabla\vert^{1/2}}

\newtheorem{theorem}{Theorem}[section]

\newtheorem{lemma}[theorem]{Lemma}

\newtheorem{proposition}[theorem]{Proposition}

\theoremstyle{definition}
\newtheorem{definition}[theorem]{Definition}
\newtheorem{remark}[theorem]{Remark}

\theoremstyle{remark}

\begin{document}

			\title[Defocusing $\dot{H}^{1/2}$-critical NLS in high dimensions]
			{The defocusing $\dot{H}^{1/2}$-critical NLS in high dimensions}
			
								\author{Jason Murphy}
			\address{Department of Mathematics, UCLA,
             					  Los Angeles, CA 90095-1555, USA}
      								  \email{murphy@math.ucla.edu}
										
									\begin{abstract}
We consider the defocusing $\dot{H}^{1/2}$-critical nonlinear Schr\"odinger equation in dimensions $d\geq 5.$ In the spirit of Kenig and Merle \cite{KM}, we combine a concentration-compactness approach with the Lin--Strauss Morawetz inequality to prove that if a solution $u$ is bounded in $\dot{H}^{1/2}$ throughout its lifespan, then $u$ is global and scatters. 
									\end{abstract}

									\maketitle
									
									\section{Introduction}
									\label{introduction section}
									
	We consider the initial-value problem for the defocusing $\dot{H}_x^{1/2}$-critical nonlinear Schr\"odinger equation in dimensions $d\geq 5$:							
						\begin{equation}\label{nls}
						\left\{ \begin{array}{ll}
						(i\partial_t+\Delta)u=\vert u\vert^{\frac{4}{d-1}}u
						\\ u(0)=u_0,
						\end{array}\right.
						\end{equation}
with $u:\R_t\times\R_x^d\to\C$. This equation is deemed $\dot{H}_x^{1/2}$-critical because the rescaling that preserves the class of solutions to \eqref{nls}, that is,
				$u(t,x)\mapsto\lambda^{\frac{d-1}{2}} u(\lambda^2t,\lambda x),$ leaves invariant the $\dot{H}_x^{1/2}$-norm of the initial data. 

	In \cite{KM}, Kenig and Merle considered \eqref{nls} with $d=3$. They proved that if a solution $u$ stays bounded in $\dot{H}_x^{1/2}$ throughout its lifespan, then $u$ must be global and scatter. The same statement for $d=4$ was proven as a special case of the results of \cite{swamprat}. In this short note, we establish this result for all $d\geq 5$.  

	We begin with some definitions.
						\begin{definition}
						[Solution]
						\label{solution definition}
A function $u:I\times\R^d\to\C$ on a time interval $I\ni 0$ is a \emph{solution} to \eqref{nls} if it belongs to 
		$C_t\dot{H}_x^{1/2}(K\times\R^d)\cap L_{t,x}^{\frac{2(d+2)}{d-1}}(K\times\R^d)$ 
for every compact $K\subset I$ and obeys the Duhamel formula
			$$u(t)=e^{it\Delta}u_0-i\int_0^t e^{i(t-s)\Delta}(\vert u\vert^{\frac{4}{d-1}}u)(s)\,ds$$
for all $t\in I$. We call $I$ the \emph{lifespan} of $u$; we say $u$ is a \emph{maximal-lifespan solution} if it cannot be extended to any strictly larger interval. If $I=\R$, we say $u$ is \emph{global}. 
						\end{definition}
						\begin{definition}[Scattering size and blowup]
	We define the \emph{scattering size} of a solution $u$ to \eqref{nls} on a time interval $I$ by
						\begin{equation}
						\label{scattering size definition}
						S_I(u):=\int_I\int_{\R^d}\vert u(t,x)\vert^{\frac{2(d+2)}{d-1}}\,dx\,dt.
						\end{equation}
						
	If there exists $t\in I$ such that $S_{[t,\sup I)}(u)=\infty$, then we say $u$ \emph{blows up forward in time}. Similarly, if there exists $t\in I$ such that $S_{(\inf I,t]}(u)=\infty$, then we say $u$ \emph{blows up backward in time}. 

	On the other hand, if $u$ is global with $S_\R(u)<\infty$, then standard arguments show that $u$ \emph{scatters}, that is, there exist unique $u_\pm\in\dot{H}_x^{1/2}(\R^d)$ such that
				$$\lim_{t\to\pm\infty}\norm{u(t)-e^{it\Delta}u_{\pm}}_{\dot{H}_x^{1/2}(\R^d)}=0.$$ 
 								\end{definition}
								
	Our main result is the following			
							\begin{theorem}\label{scattering theorem}
Let $d\geq 5$ and let $u:I\times\R^d\to\C$ be a maximal-lifespan solution to \eqref{nls} such that $u\in L_t^\infty\dot{H}_x^{1/2}(I\times\R^d).$ Then $u$ is global and scatters, with
					$$S_\R(u)\leq C\big(\norm{u}_{L_t^\infty\dot{H}_x^{1/2}(\R\times\R^d)}\big)
					\quad\text{for some function}\quad C:[0,\infty)\to[0,\infty).$$

							\end{theorem}

	Following the approach of Kenig and Merle \cite{KM}, we will establish Theorem \ref{scattering theorem} by combining a concentration-compactness argument with the Lin--Strauss Morawetz inequality of \cite{LS}. This estimate is very useful in the study of \eqref{nls}, as it has critical scaling for this problem. In fact, it is the concentration-compactness component that comprises most of this note; once we have reduced the problem to the study of almost periodic solutions, we can quickly bring the argument to a conclusion. 

	We first need a good local-in-time theory. Building off arguments of Cazenave and Weissler \cite{cw}, we can prove the following local well-posedness result (see Remark~\ref{deduce lwp}). 

						\begin{theorem}[Local well-posedness]\label{standard lwp}
Let $d\geq 5$ and $u_0\in\dot{H}_x^{1/2}(\R^d)$. Then there exists a unique maximal-lifespan solution $u:I\times\R^d\to\C$ to \eqref{nls} such that:

$\bullet$ $($Local existence$)$ $I$ is an open neighborhood of $0$.

$\bullet$ $($Blowup criterion$)$ If $\sup I<\infty$, then $u$ blows up forward in time. If $\vert\inf I\vert<\infty$, then $u$ blows up backward in time.

$\bullet$ $($Existence of wave operators$)$ For any $u_+\in\dot{H}_x^{1/2}(\R^d)$, there is a unique solution $u$ to \eqref{nls} such that $u$ scatters to $u_+$, that is,
							$$
							\lim_{t\to\infty}\norm{u(t)-e^{it\Delta}u_+}_{\dot{H}_x^{1/2}(\R^d)}=0.
							$$
A similar statement holds backward in time.

$\bullet$ $($Small-data global existence$)$ If $\norm{u_0}_{\dot{H}_x^{1/2}(\R^d)}$ is sufficiently small depending on $d$, then $u$ is global and scatters, with $S_\R(u)\lesssim\norm{u_0}_{\dot{H}_x^{1/2}(\R^d)}^{\frac{2(d+2)}{d-1}}.$ 

						\end{theorem}

					\subsection{Outline of the proof of Theorem \ref{scattering theorem}}
					\label{outline sub}
	We argue by contradiction and suppose that Theorem \ref{scattering theorem} fails. Recalling from Theorem \ref{standard lwp} that Theorem \ref{scattering theorem} holds for sufficiently small initial data, we deduce the existence of a threshold size, below which Theorem \ref{scattering theorem} holds, but above which we can find (almost) counterexamples. We then use a limiting argument to find blowup solutions \emph{at} this threshold, and show that such minimal blowup solutions must possess strong concentration properties. Finally, in Sections \ref{finite section} and \ref{lin--strauss section}, we show that solutions to \eqref{nls} with such properties cannot exist. 	

	The main property of these solutions is that of almost periodicity:
\begin{definition}[Almost periodic solutions]\label{almost periodic definition} A solution $u$ to \eqref{nls} with lifespan $I$ is said to be \emph{almost periodic }$($\emph{modulo symmetries$)$} if $u\in L_t^\infty\dot{H}_x^{1/2}(I\times\R^d)$ and there exist functions $N:I\to\R^+,$ $x:I\to\R^d$, and $C:\R^+\to \R^+$ such that
			$$\int_{\vert x-x(t)\vert\geq \frac{C(\eta)}{N(t)}}\big\vert\nsc u(t,x)\big\vert^2\, dx
			+\int_{\vert\xi\vert\geq C(\eta)N(t)}\vert\xi\vert\,\vert\wh{u}(t,\xi)\vert^2\, d\xi\leq\eta$$
for all $t\in I$ and $\eta>0$. We call $N$ the \emph{frequency scale function}, $x$ the \emph{spatial center function}, and $C$ the \emph{compactness modulus function}. 
\end{definition} 
\begin{remark}\label{arzela ascoli} Using Arzel\`a--Ascoli and Sobolev embedding, one can derive the following: for a nonzero almost periodic solution $u:I\times\R^d\to \C$ to \eqref{nls}, there exists $C(u)>0$ such that
			$$\int_{\vert x-x(t)\vert\leq\frac{C(u)}{N(t)}}\vert u(t,x)\vert^{\frac{2d}{d-1}}\,dx\gtrsim_u 1\quad
			\text{uniformly for }t\in I.$$ 
\end{remark}

	We can now describe the first major step in the proof of Theorem \ref{scattering theorem}. 
							
						\begin{theorem}
						[Reduction to almost periodic solutions]
						\label{reduction theorem}
	If Theorem \ref{scattering theorem} fails, then there exists a maximal-lifespan solution $u:I\times\R^d\to\C$ to \eqref{nls} such that $u$ is almost periodic and blows up in both time directions.			\end{theorem}

	The reduction to almost periodic solutions has become a well-known and widely used technique in the study of dispersive equations at critical regularity. The existence of such solutions was first established by Keraani \cite{Keraani:L2} in the context of the mass-critical NLS, while Kenig and Merle \cite{kenig merle} were the first to use them to prove a global well-posedness result (in the energy-critical setting). These techniques have since been adapted to a variety of settings (see \cite{hr, KM, KTV, KV, KV5, KV:supercritical, swamprat, TVZ, TVZ2} for some examples in the case of NLS), and the general approach is well-understood.
	
	The argument, which we will carry out in Section \ref{reduction section}, requires three ingredients: (i) a profile decomposition for the linear propagator, (ii) a stability result for the nonlinear equation, and (iii) a decoupling statement for nonlinear profiles. The first profile decompositions established for $e^{it\Delta}$ were adapted to the mass- and energy-critical settings (see  \cite{begout-vargas, carles-keraani, keraani, merle-vega}); the case of non-conserved critical regularity was addressed in \cite{shao}. We will be able to import the profile decomposition we need directly from \cite{shao} (see Section \ref{cc section}). 
	
	Ingredients (ii) and (iii) are closely related, in that the decoupling must be established in a space that is dictated by the stability result. Stability results most often require errors to be small in a space with the scaling-critical number of derivatives (say $s_c$). In \cite{keraani}, Keraani showed how to establish the decoupling in such a space for the energy-critical problem (that is, $s_c=1$). The argument  relies on pointwise estimates and hence is also applicable to the mass-critical problem ($s_c=0$). For $s_c\notin\{0,1\}$, however, the nonlocal nature of $\vert\nabla\vert^{s_c}$ prevents the direct use of this argument.
		
	In certain cases for which $s_c\notin\{0,1\}$ it has nonetheless been possible to adapt the arguments of \cite{keraani} to establish the decoupling in a space with $s_c$ derivatives. Kenig and Merle \cite{KM} were able to succeed in the case $s_c=1/2$, $d=3$ (for which the nonlinearity is cubic) by exploiting the polynomial nature of the nonlinearity and making use of a paraproduct estimate. Killip and Vi\c{s}an \cite{KV:supercritical} handled some cases for which $s_c>1$ by utilizing a square function of Strichartz that shares estimates with $\vert\nabla\vert^{s_c}$. In \cite{swamprat}, some cases were treated for which $s_c\in(0,1)$ (and the nonlinearity is non-polynomial) by making use of the Littlewood--Paley square function and working at the level of individual frequencies.

	In this paper, we take a simpler approach to (ii) and (iii), inspired by the work of Holmer and Roudenko \cite{hr} on the focusing $\dot{H}_x^{1/2}$-critical NLS in $d=3$. It relies on the observation that for $s_c=1/2$, one can develop a stability theory for NLS that only requires errors to be small in a space without derivatives. In Section \ref{stability section}, we do exactly this (see Theorem \ref{stability theorem}). To prove the decoupling in a space without derivatives, we can then rely simply on pointwise estimates and apply the arguments of \cite{keraani} directly (see Lemma \ref{info about unJ}). By proving a more refined stability result, we are thus able to avoid entirely the technical issues related to fractional differentiation described above. In this way, we can greatly simplify the analysis needed to carry out the reduction to almost periodic solutions in our setting.

	Continuing from Theorem \ref{reduction theorem}, we can make some further reductions to the class of solutions that we consider. In particular, we can prove the following
	
			\begin{theorem}\label{further reduction}
If Theorem \ref{scattering theorem} fails, then there exists an almost periodic solution $u:[0,T_{max})\to\C$ to \eqref{nls} with the following properties: 

$(i)$ $u$ blows up forward in time,

$(ii)$ $\inf_{t\in[0,T_{max})}N(t)\geq 1,$

$(iii)$ $\vert x(t)\vert\lesssim_u\int_0^t N(s)\,ds$ for all $t\in[0,T_{max}).$ 	
			\end{theorem}
	
	Let us briefly sketch the proof of Theorem \ref{further reduction}. Beginning with an almost periodic solution as in Theorem \ref{reduction theorem} and using a rescaling argument (as in \cite[Theorem 3.3]{TVZ2}, for example), one can deduce the existence of an almost periodic blowup solution that does not escape to artibrarily low frequencies on at least half of its maximal lifespan, say $[0,T_{max})$. In this way, we may find an almost periodic solution such that $(i)$ and $(ii)$ hold. 
	
	It remains to see how one can modify an almost periodic solution $u:[0,T_{max})\times\R^d\to\C $ so that $(iii)$ holds. We may always translate $u$ so that $x(0)=0$; thus, it suffices to show that we can modify the modulation parameters of $u$ so that $\vert\dot{x}(t)\vert\sim_u N(t)$ for a.e. $t\in[0,T_{max})$. This will follow from the following local constancy property for the modulation parameters of almost periodic solutions (see \cite[Lemma~5.18]{KV}, for example):

				\begin{lemma}[Local constancy]\label{local constancy}
Let $u:I\times\R^d\to\C$ be a maximal-lifespan almost periodic solution to \eqref{nls}. Then there exists $\delta=\delta(u)>0$ such that if $t_0\in I$, then
					$$[t_0-\delta N(t_0)^{-2},t_0+\delta N(t_0)^{-2}]\subset I,$$
					with
					$$N(t)\sim_u N(t_0),\quad\vert x(t)-x(t_0)\vert\lesssim_u N(t_0)^{-1}\quad
					\text{for}\quad\vert t-t_0\vert\leq \delta N(t_0)^{-2}.$$
			\end{lemma}	
	
	Using Lemma \ref{local constancy}, we may subdivide $[0,T_{max})$ into \emph{characteristic subintervals} $I_k$ and set $N(t)$ to be constant and equal to some $N_k$ on each $I_k$. Note that $\vert I_k\vert\sim_u N_k^{-2}$ and that this requires us to modify the compactness modulus function by a (time-independent) multiplicative factor. We may then modify $x(t)$ by $O(N(t)^{-1})$ so that $x(t)$ becomes piecewise linear on each $I_k$, with $\vert\dot{x}(t)\vert\sim_u N(t)$ for $t\in I_k^\circ$. Thus, we get $\vert\dot{x}(t)\vert\sim_u N(t)$ for a.e. $t\in [0,T_{max}),$ as desired. 		
	To complete the proof of Theorem \ref{scattering theorem}, it therefore suffices to rule out the existence of the almost periodic solutions described in Theorem \ref{further reduction}.
		
	In Section \ref{finite section}, we preclude the possibility of finite time blowup (i.e. $T_{max}<\infty$). To do this, we make use of the following `reduced' Duhamel formula for almost periodic solutions, which one can prove by adapting the argument in \cite[Proposition 5.23]{KV}:
							\begin{proposition}		
							[Reduced Duhamel formula]
							\label{reduced duhamel proposition}
Let $u:[0,T_{max})\times\R^d\to\C$ be an almost periodic solution to \eqref{nls}. Then for all $t\in[0,T_{max})$, we have
							$$
						u(t)=i\lim_{T\to T_{max}}\int_t^T e^{i(t-s)\Delta}(\vert u\vert^{\frac{4}{d-1}}u)(s)\,ds,
							$$
where the limits are taken in the weak $\dot{H}_x^{1/2}$ topology. 							
							\end{proposition}
	Using Proposition \ref{reduced duhamel proposition}, Strichartz estimates, and conservation of mass, we can show that an almost periodic solution that blows up in finite time must have zero mass, contradicting the fact that the solution blows up in the first place.

	In Section \ref{lin--strauss section}, we use the Lin--Strauss Morawetz inequality to rule out the remaining case, $T_{max}=\infty$. This estimate tells us that solutions to NLS that are bounded in $\dot{H}_x^{1/2}$ cannot remain concentrated near the origin for too long. However, almost periodic solutions to \eqref{nls} as in Theorem \ref{further reduction} with $T_{max}=\infty$ do essentially this; thus we can reach a contradiction in this case.  

\subsection*{Acknowledgements} I am grateful to my advisors Rowan Killip and Monica Vi\c{s}an for helpful discussions and for a careful reading of the manuscript. This work was supported in part by NSF grant DMS-1001531 (P.I. Rowan Killip). 
								
									\section{Notation and useful lemmas}
									\subsection{Some notation}
	We write $X\lesssim Y$ or $Y\gtrsim X$ whenever $X\leq CY$ for some $C=C(d)>0$. If $X\lesssim Y\lesssim X$, we write $X\sim Y$. Dependence on additional parameters will be indicated with subscripts, for example, $X\lesssim_u Y$.
									
	For a spacetime slab $I\times\R^d$, we define
				$$\xnorm{u}{q}{r}{I}:=\norm{\,\norm{u(t)}_{L_x^r(\R^d)}\,}_{L_t^q(I)}.$$ 
If $q=r$, we write $L_t^{q}L_x^q=L_{t,x}^q.$ We also sometimes write $\norm{f}_{L_x^r(\R^d)}=\norm{f}_{L_x^r}$. 

	We define the Fourier transform on $\R^d$ by
				$$\wh{f}(\xi):=(2\pi)^{-d/2}\int_{\R^d} e^{-ix\cdot \xi}f(x)\,dx.$$
For $s>-d/2$, we then define the fractional differentiation operator $\vert\nabla\vert^s$ and the homogeneous Sobolev norm  via $\wh{\vert\nabla\vert^s f}(\xi):=\vert\xi\vert^s\wh{f}(\xi)$ and $\norm{f}_{\dot{H}_x^{s}(\R^d)}:=\norm{\vert\nabla\vert^s f}_{L_x^2(\R^d)}.$
							\subsection{Basic harmonic analysis}
	Let $\varphi$ be a radial bump function supported in the ball $\{\xi\in\R^d:\vert \xi\vert\leq\tfrac{11}{10}\}$ and equal to 1 on the ball $\{\xi\in\R^d:\vert\xi\vert\leq 1\}.$ For $N\in 2^{\mathbb{Z}}$, we define the Littlewood--Paley projection operators via
					\begin{align*}
					&\wh{P_{\leq N} f}(\xi):=\varphi(\tfrac{\xi}{N})\wh{f}(\xi),
					\quad \wh{P_N f}(\xi):=\big(\varphi(\tfrac{\xi}{N})-\varphi(\tfrac{2\xi}{N})\big)\wh{f}(\xi),
					\quad P_{>N}:=\text{Id}-P_{\leq N}.
					\end{align*}
								
We note that these operators commute with $e^{it\Delta}$ and all differential operators, as they are Fourier multiplier operators. They also obey the following
						\begin{lemma}
						[Bernstein estimates] 
For $1\leq r\leq q\leq\infty$ and $s\geq 0$,
				\begin{align*}
 			\norm{P_{> N}f}_{L_x^r(\R^d)}\lesssim N^{-s}\norm{\vert\nabla\vert^sf}_{L_x^r(\R^d)},
			\quad \norm{P_{N}f}_{L_x^q(\R^d)}\lesssim N^{\frac{d}{r}-\frac{d}{q}}\norm{f}_{L_x^r(\R^d)}.
				\end{align*}
						\end{lemma}									
	We will also need some fractional calculus estimates.						
						\begin{lemma}[Fractional chain rule, \cite{CW}]	
Suppose $G\in C^1(\C)$ and $s\in(0,1]$. Let $1<r,r_1<\infty,$ $1<r_2\leq\infty$ be such that $\tfrac{1}{r}=\tfrac{1}{r_1}+\tfrac{1}{r_2}.$ Then 
			$$\norm{\vert\nabla\vert^s G(u)}_{L_x^r}\lesssim 
			\norm{G'(u)}_{L_x^{r_1}}\norm{\vert\nabla\vert^s u}_{L_x^{r_2}}.$$	
						\end{lemma}								
						\begin{lemma}[Derivatives of differences, \cite{KV NLW2}]
						\label{derivatives of differences}
						Let $F(u)=\vert u\vert^p u$ for some $p>0$ and let $0<s<1$. For $1<r,r_1,r_2<\infty$ such that $\tfrac{1}{r}=\tfrac{1}{r_1}+\tfrac{p}{r_2},$ we have
						$$\norm{\vert\nabla\vert^s[F(u+v)-F(u)]}_{L_x^r}
						\lesssim\norm{\vert\nabla\vert^s u}_{L_x^{r_1}}\norm{v}_{L_x^{r_2}}^p
						+\norm{\vert\nabla\vert^s v}_{L_x^{r_1}}\norm{u+v}_{L_x^{r_2}}^p.$$
						\end{lemma}									

								\subsection{Strichartz estimates}
	Let $e^{it\Delta}$ be the free Schr\"odinger propagator:
		$$[e^{it\Delta}f](x)=(4\pi i t)^{-d/2}\int_{\R^d} e^{i\vert x-y\vert^2/4t}f(y)\,dy\quad\text{for }t\neq 0.$$

	For $d\geq 3$, we call a pair of exponents $(q,r)$ \emph{Schr\"odinger admissible} if $2\leq q,r\leq\infty$ and $\tfrac{2}{q}+\tfrac{d}{r}=\tfrac{d}{2}$. For a time interval $I$ and $s\geq 0$, we define the Strichartz space $\dot{S}^s(I)$ via the norm 
		$$\norm{u}_{\dot{S}^{s}(I)}=\sup\big\{\xnorm{\vert\nabla\vert^s u}{q}{r}{I}:(q,r)\text{ Schr\"odinger admissible}\big\}.$$ 
	We will make frequent use of the following standard estimates for $e^{it\Delta}$:
						\begin{lemma}
						[Strichartz estimates, \cite{ginibre velo smoothing, keel tao, strichartz}]
Let $d\geq 3$, $s\geq 0$ and let $I$ be a compact time interval. Let $u:I\times\R^d\to\C$ be a solution to the forced Schr\"odinger equation $(i\partial_t+\Delta)u=F$. Then for any $t_0\in I$, we have
			$$\norm{u}_{\dot{S}^s(I)}\lesssim\norm{u(t_0)}_{\dot{H}_x^s(\R^d)}+\min\bigg\{\xnorm{\vert\nabla\vert^s F}{2}{\frac{2d}{d+2}}{I},
			\xnorms{\vert\nabla\vert^s F}{\frac{2(d+2)}{d+4}}{I}\bigg\}.$$ 
						\end{lemma}
									\subsection{Concentration-compactness}
									\label{cc section}
	We record here the linear profile decomposition that we will use in Section \ref{reduction section}. We begin with the following

							\begin{definition}[Symmetry group]\label{symmetry group}
For any position $x_0\in\R^d$ and scaling parameter $\lambda>0$, we define a unitary transformation $g_{x_0,\lambda}:\dot{H}_x^{1/2}(\R^d)\to\dot{H}_x^{1/2}(\R^d)$ by
			$$[g_{x_0,\lambda}f](x):=\lambda^{-\frac{d-1}{2}}f(\lambda^{-1}(x-x_0)).$$
We let $G$ denote the group of such transformations. 
							\end{definition} 	
We now state the linear profile decomposition. For the mass-critical NLS, this result was originally proven in \cite{begout-vargas, carles-keraani, merle-vega}, while for the energy-critical NLS, it was established in \cite{keraani}. In the generality we need, a proof can be found in \cite{shao}.

						\begin{lemma}
						[Linear profile decomposition, \cite{shao}]
						\label{linear profile decomposition}
Let $\{u_n\}_{n\geq 1}$ be a bounded sequence in $\dot{H}_x^{1/2}(\R^d).$ After passing to a subsequence if necessary, there exist functions $\{\phi^j\}_{j\geq 1}\subset\dot{H}_x^{1/2}(\R^d)$, group elements $g_n^j\in G$ $($with parameters $x_n^j$ and $\lambda_n^j$$)$, and times $t_n^j\in\R$ such that for all $J\geq 1$, we have the following decomposition:
				$$u_n=\sum_{j=1}^J g_n^j e^{it_n^j\Delta}\phi^j+w_n^J.$$
This decomposition satisfies the following properties:

$\bullet$ For each $j$, either $t_n^j\equiv 0$ or $t_n^j\to\pm\infty$ as $n\to\infty.$

$\bullet$ For $J\geq 1$, we have the following decoupling:
			\begin{equation}\label{decoupling}
		\lim_{n\to\infty}\bigg[ \norm{u_n}_{\dot{H}_x^{1/2}}^2-\sum_{j=1}^J\norm{\phi^j}_{\dot{H}_x^{1/2}}^2
		-\norm{w_n^J}_{\dot{H}_x^{1/2}}^2\bigg]=0.	
			\end{equation}

$\bullet$ For any $j\neq k$, we have the following asymptotic orthogonality condition:
				\begin{equation}\label{orthogonality}
		\frac{\lambda_n^j}{\lambda_n^k}+\frac{\lambda_n^k}{\lambda_n^j}
		+\frac{\vert x_n^j-x_n^k\vert^2}{\lambda_n^j\lambda_n^k}
		+\frac{\vert t_n^j(\lambda_n^j)^2-t_n^k(\lambda_n^k)^2\vert}{\lambda_n^j\lambda_n^k}
		\to\infty
		\quad\text{as }n\to\infty.
				\end{equation}
							
$\bullet$ For all $n$ and all $J\geq 1$, we have $w_n^J\in\dot{H}_x^{1/2}(\R^d)$, with
					\begin{equation}\label{remainder}
					\lim_{J\to\infty}\limsup_{n\to\infty}\xnorms{e^{it\Delta}w_n^J}{\frac{2(d+2)}{d-1}}{\R}=0.
					\end{equation}
					
\end{lemma}
									\section{Stability theory}
									\label{stability section}
	In this section, we develop a stability theory for \eqref{nls}. The main result of this section is Theorem \ref{stability theorem}, which will play a key role in the reduction to almost periodic solutions carried out in Section \ref{reduction section}. Throughout this section, we will denote the nonlinearity $\vert u\vert^{\frac{4}{d-1}}u$ by $F(u)$. 
	
	We begin by recording a local well-posedness result of Cazenave and Weissler \cite{cw}. This result requires the data to belong to the inhomogeneous Sobolev space, so that a contraction mapping argument may be run in mass-critical spaces. 
			
								\begin{theorem}[Standard local well-posedness \cite{cw}]\label{slwp} 
Let $d\geq 5$ and $u_0\in{H}_x^{1/2}(\R^d)$. If $I\ni 0$ is a time interval such that
				$$\xnorm{\nsc e^{it\Delta}u_0}{\frac{2(d+1)}{d-1}}{\frac{2d(d+1)}{d^2-d+2}}{I}$$
is sufficiently small, then we may find a unique solution $u:I\times\R^d\to\C$ to \eqref{nls}.
								\end{theorem}															
	
	Next, we turn to the stability results. We will make use of function spaces that are critical with respect to scaling, but do not involve any derivatives. In particular, for a time interval $I$, we define the following norms:
				\begin{align*}
				\norm{u}_{X(I)}:=\xnorm{u}{\frac{4(d+1)}{d-1}}{\frac{2(d+1)}{d-1}}{I},
				\quad
				\norm{F}_{Y(I)}:=\xnorm{F}{\frac{4(d+1)}{d+3}}{\frac{2(d+1)}{d+3}}{I}.
				\end{align*}
	
	We first relate the $X$-norm to the usual Strichartz norms. By Sobolev embedding, we get $\norm{u}_{X(I)}\lesssim \norm{u}_{\dot{S}^{1/2}(I)},$ while H\"older and Sobolev embedding together imply
			\begin{equation}
			\label{e and i}
			\xnorms{u}{\frac{2(d+2)}{d-1}}{I}\lesssim \norm{u}_{X(I)}^c\norm{u}_{\dot{S}^{1/2}(I)}^{1-c}\quad\text{for some}\quad 0<c(d)<1.
			\end{equation}

	Next, we record a Strichartz estimate, which one can prove via the standard approach (namely, by applying the dispersive estimate and Hardy--Littlewood--Sobolev).
				\begin{lemma}
Let $I$ be a compact time interval and $t_0\in I$. Then for all $t\in I$,
				\begin{equation}\label{xy strichartz}
				\bigg\|\int_{t_0}^t e^{i(t-s)\Delta}F(s)\,ds\bigg\|_{X(I)}\lesssim \norm{F}_{Y(I)}.
				\end{equation}
						\end{lemma}
	Finally, we collect some estimates that will allow us to control the nonlinearity.
							\begin{lemma}\label{nonlinear estimate lemma} 
Fix $d\geq 5$. Then, with spacetime norms over $I\times\R^d$, we have
	\begin{align}
	\norm{F(u)}_{Y(I)}&\lesssim\norm{u}_{X(I)}^{\frac{d+3}{d-1}}									\label{y to x}
	\\
	\norm{F(u)-F(\tilde{u})}_{Y(I)}&\lesssim
	\big\{	\norm{u}_{X(I)}^{\frac{4}{d-1}}+\norm{\tilde{u}}_{X(I)}^{\frac{4}{d-1}}\big\}
	\norm{u-\tilde{u}}_{X(I)}
								\label{y to x diff}
	\\
	\xonorm{\nsc F(u)}{2}{\frac{2d}{d+2}}
	&\lesssim \norm{u}_{X(I)}^{\frac{4}{d-1}}\norm{u}_{\dot{S}^{1/2}(I)}									\label{n to x}
	\\
	\xonorm{\nsc[F(u)-F(\tilde{u})]}{2}{\frac{2d}{d+2}}
	&\lesssim\norm{u-\tilde{u}}_{X(I)}^{\frac{4}{d-1}}\norm{\tilde{u}}_{\dot{S}^{1/2}(I)}+\norm{u}_{X(I)}^{\frac{4}{d-1}}\norm{u-\tilde{u}}_{\dot{S}^{1/2}(I)}.								\label{n to x diff}
	\end{align}	
								\end{lemma}
\begin{proof} We first note that \eqref{y to x} follows from H\"older, while \eqref{y to x diff} follows from the fundamental theorem of calculus followed by H\"older. 

Next, we see that \eqref{n to x} follows from H\"older and the fractional chain rule. Indeed,
		$$\xonorm{\nsc F(u)}{2}{\frac{2d}{d+2}}\lesssim\xonorm{u}{\frac{4(d+1)}{d-1}}{\frac{2(d+1)}{d-1}}^{\frac{4}{d-1}}
								\xonorm{\nsc u}{\frac{2(d+1)}{d-1}}{\frac{2d(d+1)}{d^2-d+2}}.$$
Using these same exponents with Lemma \ref{derivatives of differences}, we deduce \eqref{n to x diff}. 
\end{proof}	
	
	We may now state our first stability result.
	
								\begin{lemma}
								[Short-time perturbations]
								\label{short-time}
Let $d\geq 5$ and let $I$ be a compact time interval, with $t_0\in I$. Let $\tilde{u}:I\times\R^d\to\C$ be a solution to $(i\partial_t+\Delta)\tilde{u}=F(\tilde{u})+e$ with $\tilde{u}(t_0)=\tilde{u}_0\in\dot{H}_x^{1/2}$. Suppose 
						\begin{equation}\label{finite energy}
						\norm{\tilde{u}}_{\dot{S}^{1/2}(I)}\leq E\quad\text{and}\quad\xnorms{\nsc e}{\frac{2(d+2)}{d+4}}{I}\leq E
						\end{equation}		
				%
for some $E>0$. Let $u_0\in\dot{H}_x^{1/2}(\R^d)$ satisfy
				\begin{align}
				\norm{u_0-\tilde{u}_0}_{\dot{H}_x^{1/2}}&\leq E,	\label{data Eness}
				\end{align}
and suppose that we have the smallness conditions
				\begin{align}
				\norm{\tilde{u}}_{X(I)}&\leq\delta, \label{delta smallness}
				\\ \norm{e^{i(t-t_0)\Delta}(u_0-\tilde{u}_0)}_{X(I)}+\norm{e}_{Y(I)}&\leq\eps, \label{data smallness}
				\end{align}
for some small $0<\delta=\delta(E)$ and $0<\eps<\eps_0(E).$ Then there exists $u:I\times\R^d\to\C$ solving \eqref{nls} with $u(t_0)=u_0$ such that
				\begin{align}
				\norm{u-\tilde{u}}_{X(I)}+\norm{F(u)-F(\tilde{u})}_{Y(I)}&\lesssim\eps,						
				\label{X close}
				\\ 
				\norm{u-\tilde{u}}_{\dot{S}^{1/2}(I)}+
				\xnorm{\nsc[F(u)-F(\tilde{u})]}{2}{\frac{2d}{d+2}}{I}
				&\lesssim_E 1.
				\label{S difference bounded}
				\end{align}
								\end{lemma}
	
\begin{proof}
We first suppose $u_0\in L_x^2$, so that Theorem \ref{slwp} provides the solution $u$. We will then prove \eqref{X close} and \eqref{S difference bounded} as \emph{a priori} estimates. After the lemma is proven for $u_0\in H_x^{1/2}$, we can use approximation by $H_x^{1/2}$-functions to see that the lemma holds for $u_0\in\dot{H}_x^{1/2}$. Throughout the proof, spacetime norms will be over $I\times\R^d$.

We will first show
					\begin{equation}\label{u delta smallness}
					\norm{u}_{X(I)}\lesssim\delta.
					\end{equation}
By the triangle inequality, \eqref{xy strichartz}, \eqref{y to x}, \eqref{delta smallness}, and \eqref{data smallness}, we get
	\begin{align*}
	\norm{e^{i(t-t_0)\Delta}\tilde{u}_0}_{X(I)}\lesssim \norm{\tilde{u}}_{X(I)}+\norm{F(\tilde{u})}_{Y(I)}+\norm{e}_{Y(I)}
	\lesssim\delta+\delta^{\frac{d+3}{d-1}}+\eps.
	\end{align*}
Combining this estimate with \eqref{data smallness} and using the triangle inequality then gives
		$$\norm{e^{i(t-t_0)\Delta}u_0}_{X(I)}\lesssim\delta$$
for $\delta$ and $\eps\lesssim\delta$ sufficiently small. Thus, by \eqref{xy strichartz} and \eqref{y to x}, we get
		\begin{align*}
		\norm{u}_{X(I)}\lesssim \delta+\norm{F(u)}_{Y(I)}
			\lesssim\delta+\norm{u}_{X(I)}^{\frac{d+3}{d-1}},
		\end{align*}
which (taking $\delta$ sufficiently small) implies \eqref{u delta smallness}.

	We now turn to proving the desired estimates for $w:=u-\tilde{u}$. Note first that $w$ is a solution to $(i\partial_t+\Delta)w=F(u)-F(\tilde{u})-e$, with $w(t_0)=u_0-\tilde{u}_0$; thus, we can use \eqref{xy strichartz}, \eqref{y to x diff}, \eqref{delta smallness}, \eqref{data smallness}, and \eqref{u delta smallness} to see
	\begin{align*}
	\norm{w}_{X(I)}&\lesssim\norm{e^{i(t-t_0)\Delta}(u_0-\tilde{u}_0)}_{X(I)}+\norm{e}_{Y(I)}+\norm{F(u)-F(\tilde{u})}_{Y(I)}
			\\	&\lesssim \eps+\big\{\norm{u}_{X(I)}^{\frac{4}{d-1}}+\norm{\tilde{u}}_{X(I)}^{\frac{4}{d-1}}\big\}\norm{w}_{X(I)}
			\\	& \lesssim \eps+\delta^{\frac{4}{d-1}}\norm{w}_{X(I)}.
	\end{align*}
Taking $\delta$ sufficiently small, we see that the first estimate in \eqref{X close} holds. Using the first estimate in \eqref{X close}, along with \eqref{y to x diff}, \eqref{delta smallness}, and \eqref{u delta smallness}, we see that the remaining estimate in \eqref{X close} holds, as well.

Next, by Strichartz, \eqref{n to x diff},  \eqref{finite energy}, \eqref{data Eness}, \eqref{delta smallness}, \eqref{X close}, and \eqref{u delta smallness}, we get
	\begin{align*}
	\norm{w}_{\dot{S}^{1/2}(I)}&\lesssim\norm{u_0-\tilde{u}_0}_{\dot{H}_x^{1/2}}+\norm{\nsc e}_{L_{t,x}^{\frac{2(d+2)}{d+4}}}+
							\norm{\nsc[F(u)-F(\tilde{u})]}_{L_t^2L_x^{\frac{2d}{d+2}}}
	\\ &\lesssim_E 1+\norm{\tilde{u}}_{\dot{S}^{1/2}(I)}\norm{w}_{X(I)}^{\frac{4}{d-1}}+\norm{w}_{\dot{S}^{1/2}(I)}\norm{u}_{X(I)}^{\frac{4}{d-1}}
	\\ &\lesssim_E 1+\delta^{\frac{4}{d-1}}\norm{w}_{\dot{S}^{1/2}(I)}.
	\end{align*}
Taking $\delta=\delta(E)$ sufficiently small then gives the first estimate in \eqref{S difference bounded}. We get the remaining estimate in \eqref{S difference bounded} by using \eqref{n to x diff} with \eqref{finite energy}, \eqref{X close}, \eqref{u delta smallness}, and the first estimate in \eqref{S difference bounded}. 
\end{proof}

\begin{remark}
As mentioned in the introduction, the error $e$ is only required to be small in a space without derivatives (see \eqref{data smallness}); it merely needs to be \emph{bounded} in a space with derivatives (see \eqref{finite energy}). This will also be the case in Theorem \ref{stability theorem} below (see \eqref{**} and \eqref{***}).  We will see the benefit of this refinement when we carry out the proof of Theorem \ref{reduction theorem} in Section \ref{reduction section} (see Remark \ref{talk}). 
\end{remark}
	
	We continue to the main result of this section:
					\begin{theorem}
					[Stability]
					\label{stability theorem}
Let $d\geq 5$, and let $I$ be a compact time interval, with $t_0\in I$. Suppose $\tilde{u}$ is a solution to $(i\partial_t+\Delta)\tilde{u}=F(\tilde{u})+e$, with $\tilde{u}(t_0)=\tilde{u}_0$. Suppose
					\begin{align}
					\label{**}
					\norm{\tilde{u}}_{\dot{S}^{1/2}(I)}\leq E\quad\text{and}\quad\xnorms{\nsc e}{\frac{2(d+2)}{d+4}}{I}&\leq E
					\end{align}
for some $E>0$. Let $u_0\in\dot{H}_x^{1/2}(\R^d)$, and suppose we have the smallness conditions
					\begin{align}
					\label{***}
					\norm{u_0-\tilde{u}_0}_{\dot{H}_x^{1/2}(\R^d)}
					+\norm{e}_{Y(I)}&\leq\eps
					\end{align}
for some small $0<\eps<\eps_1(E)$. Then, there exists $u:I\times\R^d\to\C$ solving \eqref{nls} with $u(t_0)=u_0$, and there exists $0<c(d)<1$ such that
					\begin{align}
					\xnorms{u-\tilde{u}}{\frac{2(d+2)}{d-1}}{I}&\lesssim_E \eps^c.
					\label{concluuu}
					\end{align}
					\end{theorem}

One can derive Theorem \ref{stability theorem} from Lemma \ref{short-time} in the standard fashion, namely, by applying Lemma \ref{short-time} inductively (see \cite{KV}, for example). We omit these details, but pause to point out the following: this induction will actually yield the bounds
			$$\norm{u-\tilde{u}}_{X(I)}\lesssim\eps
			\quad\text{and}
			\quad\norm{u-\tilde{u}}_{\dot{S}^{1/2}(I)}\lesssim_E 1.$$
With these bounds in hand, we then use \eqref{e and i} to see that \eqref{concluuu} holds.

\begin{remark}\label{yay me}
The smallness condition on $u_0-\tilde{u}_0$ appearing in \eqref{***} may actually be relaxed to the condition appearing in \eqref{data smallness}. In our setting, it will not be difficult to prove the stronger condition (see Lemma \ref{info about unJ}).  
\end{remark}

\begin{remark}
\label{deduce lwp}
Using arguments from \cite{cazenave book, CW}, one can establish Theorem \ref{standard lwp} for data in the imhomogeneous Sobolev space $H_x^{1/2}$. Using Theorem \ref{stability theorem}, one can then remove the assumption $u_0\in L_x^2$ \emph{a posteriori} (by approximating $u_0\in\dot{H}_x^{1/2}$ by ${H}_x^{1/2}$- functions). We omit the standard details.
\end{remark}		
	
					\section{Reduction to almost periodic solutions}
					\label{reduction section}
In this section, we sketch a proof of Theorem \ref{reduction theorem}. As described in the introduction, the key ideas come from \cite{keraani, Keraani:L2} and are well-known. Thus, we will merely outline the argument, providing full details only when our approach deviates from the usual one. We model our presentation primarily after \cite[Section 3]{KV5}. Throughout this section, we denote the nonlinearity $\vert u\vert^{\frac{4}{d-1}}u$ by $F(u)$.

We suppose that Theorem \ref{scattering theorem} fails. We then define $L:[0,\infty)\to[0,\infty]$ by
			$$L(E):=\sup\big\{S_I(u)\ \big\vert\ u:I\times\R^d\to\C\text{ solving }
			\eqref{nls}\text{ with }\norm{u}_{L_t^\infty\dot{H}_x^{1/2}(I\times\R^d)}^2
			\leq E\big\}.$$
We note that $L$ is non-decreasing, with
					$L(E)\lesssim E^{\frac{d+2}{d-1}}$ for $E$ sufficiently small
(cf. Theorem \ref{standard lwp}). Thus, there exists a unique critical threshold $E_c\in(0,\infty]$ such that $L(E)<\infty$ for $E<E_c$ and $L(E)=\infty$ for $E>E_c$. The failure of Theorem \ref{scattering theorem} implies that $0<E_c<\infty.$ 

	The key to proving Theorem \ref{reduction theorem} is the following convergence result. With this result in hand, establishing Theorem \ref{reduction theorem} is a straightforward exercise (see \cite[Section 3.2]{KV5}). 

					\begin{proposition}
					[Palais--Smale condition modulo symmetries]
					\label{palais--smale}
					
Let $d\geq 5$ and let $u_n:I_n\times\R^d\to\C$ be a sequence of solutions to \eqref{nls} such that
					$$\limsup_{n\to\infty}\norm{u_n}_{L_t^\infty\dot{H}_x^{1/2}(I_n\times\R^d)}^2=E_c.$$
Suppose $t_n\in I_n$ are such that
					\begin{equation}
					\label{blowing up}
					\lim_{n\to\infty} S_{[t_n,\sup I_n)}(u_n)=\lim_{n\to\infty} S_{(\inf I_n,t_n]}(u_n)=\infty.
					\end{equation}
Then, $\{u_n(t_n)\}$ converges along a subsequence in $\dot{H}_x^{1/2}(\R^d)/G$.\end{proposition}

\begin{proof}
We first translate so that each $t_n=0$ and apply Lemma \ref{linear profile decomposition} to write
								\begin{equation}
								\label{decomposition}
								u_n(0)=\sum_{j=1}^J g_n^j e^{it_n^j\Delta}\phi^j+w_n^J
								\end{equation}
along some subsequence. Recall that for each $j$, either $t_n^j\equiv 0$ or $t_n^j\to\pm\infty$. To prove Proposition \ref{palais--smale}, we need to show that there is exactly one profile $\phi^1$, with $t_n^1\equiv 0$ and $\norm{w_n^1}_{\dot{H}_x^{1/2}}\to 0$. 

First, using Theorem \ref{standard lwp}, for each $j$ we define $v^j:I^j\times\R^d\to\C$ to be the maximal-lifespan solution to \eqref{nls} such that
							$$\left\{
							\begin{array}{ll}
							v^j(0)=\phi^j & \text{if }t_n^j\equiv 0,
							\\ v^j\text{ scatters to }\phi^j\text{ as }t\to\pm\infty & \text{if }t_n^j\to\pm\infty.
							\end{array}
							\right.
							$$
Next, we define nonlinear profiles $v_n^j:I_n^j\times\R^d\to\C$ by
				$$v_n^j(t)=g_n^jv^j\big((\lambda_n^j)^{-2}t+t_n^j\big),
				\quad\text{where}\quad I_n^j=\{t:(\lambda_n^j)^{-2}t+t_n^j\in I^j\}.$$ 			
To complete the proof, we need the following three claims:

(i) There is at least one `bad' profile $\phi^{j}$, in the sense that 
					\begin{equation}
					\label{bad bad bad}
					\limsup_{n\to\infty} S_{[0,\sup I_n^{j})}(v_n^{j})=\infty.
					\end{equation}

(ii) There can then be at \emph{most} one profile (which we label $\phi^1$), and $\norm{w_n^1}_{\dot{H}_x^{1/2}}\to 0$. 

(iii) We have $t_n^1\equiv 0$.

We will provide a proof of (i) below. The proofs of (ii) and (iii) require only small variations of the analysis given for (i), so we will merely outline the arguments here. For (ii), one can adapt the argument of \cite[Lemma 3.3]{KV5} to show that the decoupling \eqref{decoupling} persists in time (this is not obvious, as the $\dot{H}_x^{1/2}$-norm is not a conserved quantity for \eqref{nls}). The critical nature of $E_c$ may then be used to preclude the possibility of multiple profiles (and to show $\norm{w_n^1}_{\dot{H}_x^{1/2}}\to 0$). For (iii), we only need to rule out the cases $t_n^1\to\pm\infty.$ To do this, one can argue by contradiction: if $t_n^1\to\pm\infty$, one can use the stability result Theorem \ref{stability theorem} (comparing $u_n$ to $e^{it\Delta}u_n(0)$) to contradict \eqref{blowing up}. See \cite[p. 391]{KV5} for more details.  

We now turn to the proof of (i). We first note that the decoupling \eqref{decoupling} implies that the $v_n^j$ are global and scatter for $j$ sufficiently large, say for $j\geq J_0$; indeed, for $j$ sufficiently large, the $\dot{H}_x^{1/2}$-norm of $\phi^j$ must be below the small-data threshold described in Theorem \ref{standard lwp}. Thus, we need to show that there is at least one bad profile $\phi^j$ (in the sense of \eqref{bad bad bad}) in the range $1\leq j<J_0$. 

Suppose towards a contradiction that there are no bad profiles. By the blowup criterion of Theorem \ref{standard lwp}, this immediately implies that $\sup I_n^j=\infty$ for all $j$ and for all $n$ sufficiently large. In fact, we claim that we have the following:						
							\begin{equation}
							\label{summable strichartz norm}
							\limsup_{J\to\infty}\limsup_{n\to\infty} 
							\sum_{j=1}^J\norm{v_n^j}_{\dot{S}^{1/2}([0,\infty))}^2\lesssim_{E_c} 1.
							\end{equation}
Indeed, for $\eta>0$, the decoupling \eqref{decoupling} implies the existence of $J_1=J_1(\eta)$ such that 
				$$\sum_{j>J_1}\norm{\phi^j}_{\dot{H}_x^{1/2}}^2\lesssim\eta.$$
Thus, choosing $\eta$ smaller than the small-data threshold, Strichartz and a standard bootstrap argument give
				$$\sum_{j>J_1}\norm{v_n^j}_{\dot{S}^{1/2}([0,\infty))}^2\lesssim
				\sum_{j>J_1}\norm{\phi^j}_{\dot{H}_x^{1/2}}^2\lesssim\eta.$$
As the $v_n^j$ satisfy $S_{[0,\infty)}(v_n^j)\lesssim 1$ for $n$ large, we may use Strichartz and another bootstrap argument to see $\norm{v_n^j}_{\dot{S}^{1/2}}\lesssim1$ for $1\leq j\leq J_1$ and $n$ large. Thus, we conclude that \eqref{summable strichartz norm} holds.

Next, using the fact that there are no bad profiles, together with the orthogonality condition \eqref{orthogonality}, one can use the arguments of \cite{keraani} to arrive at the following
							\begin{lemma}
							[Orthogonality]
							\label{orthogonality lemma}
							For $j\neq k$, we have
		\begin{align}
			\bigg[\xonorms{v_n^jv_n^k&}{\frac{d+2}{d-1}}+\xonorm{v_n^j v_n^k}{\frac{4(d+1)}{d+3}}{\frac{2d(d+1)}{2d^2-d-5}}			
			+\xonorms{(\nsc v_n^j)(\nsc v_n^k)}{\frac{d+2}{d}} \nonumber
			\\ &+ \label{o}
			\xonorms{\big(\nsc F(v_n^j)\big)\big(\nsc F(v_n^k)\big)^{\frac{d}{d+4}}}{1}\bigg]\to 0\quad\text{as}\quad n\to\infty,
		\end{align}
where all spacetime norms are taken over $[0,\infty)\times\R^d$. 
							\end{lemma}

We now wish to use \eqref{summable strichartz norm} and \eqref{o}, together with Theorem \ref{stability theorem}, to deduce a bound on the scattering size of the $u_n$, thus contradicting \eqref{blowing up}. To this end, we define approximate solutions to \eqref{nls} and collect the information we need about them in the following 
								\begin{lemma}
								\label{info about unJ}
The approximate solutions $u_n^J(t):=\sum_{j=1}^J v_n^j(t)+e^{it\Delta}w_n^J$ satisfy
			\begin{align}					
			&\limsup_{J\to\infty}\limsup_{n\to\infty} 
			\norm{u_n(0)-u_n^J(0)}_{\dot{H}_x^{1/2}}=0,						\label{unJ0}
			\\
			&\limsup_{J\to\infty}\limsup_{n\to\infty} 
			S_{[0,\infty)}(u_n^J)\lesssim_{E_c} 1,						\label{unJ scattering size}
			\\
			&\limsup_{J\to\infty}\limsup_{n\to\infty}  
			\norm{u_n^J}_{\dot{S}^{1/2}([0,\infty))}\lesssim_{E_c} 1.			\label{unJ strichartz norm}
			\end{align}
The errors
			$e_n^J:=(i\partial_t+\Delta)u_n^J-F(u_n^J)=\sum_{j=1}^J F(v_n^j)-F(u_n^J)$
satisfy		
				\begin{align}							
				&\limsup_{J\to\infty}\limsup_{n\to\infty}
				\xnorms{\nsc e_n^J}{\frac{2(d+2)}{d+4}}{[0,\infty)}\lesssim_{E_c} 1,			\label{enJ bounded}
				\\ 
				&\limsup_{J\to\infty}\limsup_{n\to\infty}\xnorm{e_n^J}{\frac{4(d+1)}{d+3}}{\frac{2(d+1)}{d+3}}{[0,\infty)}=0.	 \label{enJ to zero11}
				\end{align}
									\end{lemma}
\begin{remark}\label{talk}
It is here that we see the benefit of the refined stability result Theorem~\ref{stability theorem}. In particular, to apply Theorem \ref{stability theorem} we only need to exhibit smallness of the $e_n^J$ in the space appearing in \eqref{enJ to zero11}. As this space contains no derivatives, we can achieve this simply by relying on pointwise estimates.
\end{remark}			
			
\begin{proof}
We first note that \eqref{unJ0} follows from the construction of the $v^j$.

Next, we turn to \eqref{unJ scattering size}. To begin, we notice that by Sobolev embedding and the fact that $\frac{2(d+2)}{d-1}\geq 2$, we may deduce from \eqref{summable strichartz norm} that
							$$
							\sum_{j\geq 1} S_{[0,\infty)}(v_n^j)\lesssim_{E_c} 1.
							$$
Thus, recalling \eqref{remainder}, we see that to prove \eqref{unJ scattering size} it will suffice to show
				\begin{equation}
				\label{scattering difference}
				\limsup_{J\to\infty}\limsup_{n\to\infty}
				\bigg\vert S_{[0,\infty)}\bigg(\sum_{j=1}^J v_n^j\bigg)-\sum_{j=1}^JS_{[0,\infty)}(v_n^j)\bigg\vert
				=0.
				\end{equation}
To this end, we fix $J$ and use H\"older's inequality, \eqref{summable strichartz norm}, and \eqref{o} to see
	\begin{align*}
	\bigg\vert S_{[0,\infty)}\bigg(\sum_{j=1}^J v_n^j\bigg)-\sum_{j=1}^JS_{[0,\infty)}(v_n^j)\bigg\vert
	&\lesssim_J \sum_{j\neq k}\xonorms{v_n^j}{\frac{2(d+2)}{d-1}}^{\frac{6}{d-1}}\xonorms{v_n^jv_n^k}{\frac{d+2}{d-1}}
	\to 0
	\end{align*}
as $n\to\infty$. This establishes \eqref{scattering difference} and completes the proof of \eqref{unJ scattering size}.

Let us next turn to \eqref{enJ bounded} (which we will later use in the proof of \eqref{unJ strichartz norm}). To begin, we will derive the bound
					\begin{equation}
					\label{a unJ strichartz bound}
					\limsup_{J\to\infty}\limsup_{n\to\infty}
					\xonorms{\nsc u_n^J}{\frac{2(d+2)}{d}}^\frac{2(d+2)}{d}\lesssim_{E_c} 1.
					\end{equation}
As $w_n^J\in \dot{H}_x^{1/2},$ it will suffice to show 
					\begin{equation}
					\label{a nice sum of strichartz norms}
					\limsup_{J\to\infty}\limsup_{n\to\infty}
					\xonorms{\sum_{j=1}^J\nsc v_n^j}{\frac{2(d+2)}{d}}^\frac{2(d+2)}{d}\lesssim_{E_c} 1.
					\end{equation} 
To this end, we first note that as $\frac{2(d+2)}{d}\geq 2$, we may use \eqref{summable strichartz norm} to see 
					\begin{equation}
					\label{a nicer sum}
					\limsup_{J\to\infty}\limsup_{n\to\infty}
					\sum_{j=1}^J\xonorms{\nsc v_n^j}{\frac{2(d+2)}{d}}^{\frac{2(d+2)}{d}}\lesssim_{E_c} 1.
					\end{equation}
On the other hand, for fixed $J$, we can use \eqref{summable strichartz norm} and \eqref{o} to see
\begin{align*}
	\bigg\vert \xonorms{&\sum_{j=1}^J \nsc v_n^j}{\frac{2(d+2)}{d}}^{\frac{2(d+2)}{d}}
	-\sum_{j=1}^J\xonorms{\nsc v_n^j}{\frac{2(d+2)}{d}}^{\frac{2(d+2)}{d}}\bigg\vert
	\\ &\lesssim_J\sum_{j\neq k} \xonorms{\nsc v_n^j}{\frac{2(d+2)}{d}}^{\frac{4}{d}}
	\xonorms{(\nsc v_n^j)(\nsc v_n^k)}{\frac{d+2}{d}}
	\to 0\quad\text{as}\quad n\to\infty.
	\end{align*}
Then \eqref{a nicer sum} implies \eqref{a nice sum of strichartz norms}, which in turn gives \eqref{a unJ strichartz bound}.

Next, by the fractional chain rule, \eqref{unJ scattering size}, and \eqref{a unJ strichartz bound}, we get
				\begin{align}
				 \xonorms{\nsc F(u_n^J)}{\frac{2(d+2)}{d+4}}
				 \lesssim
				\xonorms{u_n^J}{\frac{2(d+2)}{d-1}}^{\frac{4}{d-1}}\xonorms{\nsc u_n^J}{\frac{2(d+2)}{d}}\lesssim_{E_c}1
				\label{enJ part one}
				\end{align}  
as $n,J\to\infty$, which handles one of the terms appearing in \eqref{enJ bounded}. 

To complete the proof of \eqref{enJ bounded}, it remains to show
				$$
				\limsup_{J\to\infty}\limsup_{n\to\infty}
				\xonorms{\sum_{j=1}^J \nsc F(v_n^j)}{\frac{2(d+2)}{d+4}}^{\frac{2(d+2)}{d+4}}\lesssim_{E_c} 1.
				$$
We claim it will suffice to establish 
						\begin{equation}
						\label{unJ second piece}
						\lim_{J\to\infty}\limsup_{n\to\infty}
						\sum_{j=1}^J \xonorms{\nsc F(v_n^j)}{\frac{2(d+2)}{d+4}}^{\frac{2(d+2)}{d+4}}\lesssim_{E_c} 1.
						\end{equation}
Indeed, for fixed $J$, we have by \eqref{o} 
	\begin{align*}
	\bigg\vert \xonorms{\sum_{j=1}^J& \nsc F(v_n^j)}{\frac{2(d+2)}{d+4}}^{\frac{2(d+2)}{d+4}}
	-\sum_{j=1}^J \xonorms{\nsc F(v_n^j)}{\frac{2(d+2)}{d+4}}^{\frac{2(d+2)}{d+4}}\bigg\vert
	\\ &\lesssim_J \sum_{j\neq k}\xonorms{ \nsc F(v_n^j)\, \vert\nsc F(v_n^k)\vert^{\frac{d}{d+4}} }{1}\to 0
	\quad\text{as }n\to\infty.
	\end{align*}
	
To establish \eqref{unJ second piece} and thereby complete the proof of \eqref{enJ bounded}, we use the fractional chain rule and Sobolev embedding to see			
	\begin{align*}
	\sum_{j=1}^J \xonorms{\nsc F(v_n^j)}{\frac{2(d+2)}{d+4}}^{\frac{2(d+2)}{d+4}}
	&\lesssim\sum_{j=1}^J\big(\xonorms{v_n^j}{\frac{2(d+2)}{d-1}}^{\frac{4}{d-1}}\xonorms{\nsc v_n^j}{\frac{2(d+2)}{d}}\big)^{\frac{2(d+2)}{d+4}}
	\\ &\lesssim \sum_{j=1}^J\norm{v_n^j}_{\dot{S}^{1/2}}^{\frac{2(d+2)(d+3)}{(d+4)(d-1)}}.
	\end{align*} 		
Then \eqref{unJ second piece} follows from \eqref{summable strichartz norm} and the fact that $\frac{2(d+2)(d+3)}{(d+4)(d-1)}\geq 2$.

Now  \eqref{unJ strichartz norm} follows from an application of Strichartz, \eqref{enJ bounded} and \eqref{enJ part one}.

	It remains to establish \eqref{enJ to zero11}. We begin by rewriting
			$$
			e_n^J=\bigg[\sum_{j=1}^J F(v_n^j)-F\bigg(\sum_{j=1}^J v_n^j\bigg)\bigg]	
			+\big[F(u_n^J-e^{it\Delta}w_n^J)-F(u_n^J)\big]=:(e_n^J)_1+(e_n^J)_2.						
			$$

	We first fix $J$ and use H\"older, Sobolev embedding, \eqref{summable strichartz norm}, and \eqref{o} to see
		\begin{align*}
		\xonorm{(e_n^J)_1}{\frac{4(d+1)}{d+3}}{\frac{2(d+1)}{d+3}}
		&\lesssim_J \sum_{j\neq k} \xonorm{\, \vert v_n^j v_n^k\vert^{\frac{4}{d-1}}
							\vert v_n^j\vert^{\frac{d-5}{d-1}}\,}{\frac{4(d+1)}{d+3}}{\frac{2(d+1)}{d+3}} \nonumber
		\\ &\lesssim_J \sum_{j\neq k}\xonorm{v_n^j}{\frac{4(d+1)}{d+3}}{\frac{2d(d+1)}{d^2-d-4}}^{\frac{d-5}{d-1}}
				\xonorm{v_n^j v_n^k}{\frac{4(d+1)}{d+3}}{\frac{2d(d+1)}{2d^2-d-5}}^{\frac{4}{d-1}}\to 0
		\end{align*}
as $n\to\infty$. Next, we note that we have the pointwise estimate
			$$\vert(e_n^J)_2\vert\lesssim \vert e^{it\Delta}w_n^J\vert
			\vert f_n^J\vert^{\frac{4}{d-1}},$$
where $f_n^J:= u_n^J+ e^{it\Delta}w_n^J$ satisfies
				$\norm{f_n^J}_{\dot{S}^{1/2}}\lesssim_{E_c} 1$
as $n,J\to\infty$
(cf. \eqref{unJ strichartz norm} and the fact that $w_n^J\in\dot{H}_x^{1/2}$). Thus, we can use H\"older, Strichartz, Sobolev embedding, $w_n^J\in\dot{H}_x^{1/2}$, and \eqref{remainder} to see
	\begin{align*}
	\xonorm{(e_n^J)_2}{\frac{4(d+1)}{d+3}}{\frac{2(d+1)}{d+3}}
	&\lesssim\xonorm{e^{it\Delta}w_n^J}{\frac{4(d+1)}{d-1}}{\frac{2(d+1)}{d-1}}
	\xonorm{f_n^J}{\frac{4(d+1)}{d-1}}{\frac{2(d+1)}{d-1}}^{\frac{4}{d-1}} \nonumber
	\\ &\lesssim \xonorms{e^{it\Delta}w_n^J}{\frac{2(d+2)}{d-1}}^{\frac{d+2}{2(d+1)}}\norm{w_n^J}_{\dot{H}_x^{1/2}}^{\frac{d}{2(d+1)}}\norm{f_n^J}_{\dot{S}^{1/2}}^{\frac{4}{d-1}}\to 0\quad\text{as }n,J\to\infty.
	\end{align*}

Combining the estimates for $(e_n^J)_1$ and $(e_n^J)_2$, we conclude that \eqref{enJ to zero11} holds. 
\end{proof}

Using Lemma \ref{info about unJ}, we may apply Theorem \ref{stability theorem} to deduce that $S_{[0,\infty)}(u_n)\lesssim_{E_c} 1$ for $n$ large, contradicting \eqref{blowing up}. We conclude that there is at least one bad profile, that is, claim (i) holds. This completes the proof of Proposition~\ref{palais--smale} and Theorem~\ref{reduction theorem}. 
\end{proof}									
									\section{Finite time blowup}\label{finite section}
	In this section, we use Proposition \ref{reduced duhamel proposition}, Strichartz estimates, and conservation of mass to preclude the existence of almost periodic solutions as in Theorem \ref{further reduction} with $T_{max}<\infty$. 

						\begin{theorem}[No finite time blowup]
						\label{finite time blowup}
Let $d\geq 5$. There are no almost periodic solutions $u:[0,T_{max})\times\R^d\to\C$ to \eqref{nls} with $T_{max}<\infty$ and $S_{[0,T_{max})}=\infty.$
					\end{theorem}

\begin{proof}
Suppose that $u$ were such a solution. Then, for $t\in [0,T_{max})$ and $N>0$, Proposition \ref{reduced duhamel proposition}, Strichartz, H\"older, Bernstein, and Sobolev embedding give
	\begin{align*}
	\norm{P_Nu(t)}_{L_x^2}&\lesssim\xnorm{P_N(\vert u\vert^{\frac{4}{d-1}}u)}{2}{\frac{2d}{d+2}}{[t,T_{max})}
	\\ &\lesssim(T_{max}-t)^{1/2}N^{1/2}\xonorm{\vert u\vert^{\frac{4}{d-1}}u}{\infty}{\frac{2d}{d+3}}
	\\ &\lesssim(T_{max}-t)^{1/2}N^{1/2}\norm{u}_{L_t^\infty\dot{H}_x^{1/2}}^{\frac{d+3}{d-1}}.
	\end{align*}
As $u\in L_t^\infty\dot{H}_x^{1/2}$, we deduce
	\begin{align}
	\label{mass1}
	\norm{P_{\leq N}u(t)}_{L_x^2}\lesssim_u (T_{max}-t)^{1/2}N^{1/2}\quad\text{for all }t\in I \text{ and }N>0. 
	\end{align}
On the other hand, an application of Bernstein gives
	\begin{align}
	\label{mass2}
	\norm{P_{>N}u}_{L_t^\infty L_x^2}\lesssim N^{-1/2}\norm{u}_{L_t^\infty\dot{H}_x^{1/2}}\lesssim_u N^{-1/2}\quad\text{for all }N>0.
	\end{align}
	
	We now let $\eta>0$. We choose $N$ large enough that $N^{-1/2}<\eta$, and subsequently choose $t$ close enough to $T_{max}$ that $(T_{max}-t)^{1/2} N^{1/2}<\eta.$ Combining \eqref{mass1} and \eqref{mass2}, we then get
		$\norm{u(t)}_{L_x^2}\lesssim_u \eta.$ 
		
		As $\eta$ was arbitrary and mass is conserved, we conclude $\norm{u(t)}_{L_x^2}=0$ for all $t\in[0,T_{max})$. Thus $u\equiv 0$, which contradicts the fact that $u$ blows up.			
\end{proof}

								\section{The Lin--Strauss Morawetz inequality}
								\label{lin--strauss section}
In this section, we use the Lin--Strauss Morawetz inequality to preclude the existence of almost periodic solutions as in Theorem \ref{further reduction} such that $T_{max}=\infty.$						
								
								\begin{proposition}[Lin--Strauss Morawetz inequality, \cite{LS}]
								\label{lin--strauss}
Let $d\geq 3$ and let $u:I\times\R^d\to\C$ be a solution to $(i\partial_t+\Delta)u=\vert u\vert^p u$. Then 
								\begin{equation}\label{lsi}
								\int_I\int_{\R^d}
								\frac{\vert u(t,x)\vert^{p+2}}{\vert x\vert}
								\,dx\,dt\lesssim\norm{u}_{L_t^\infty\dot{H}_x^{1/2}(I\times\R^d)}^2.
								\end{equation}
								\end{proposition}

As in \cite{KM}, we will use this estimate to establish the following	
								\begin{theorem}
								\label{infinite time}
Let $d\geq 5$. There are no almost periodic solutions $u:[0,\infty)\times\R^d\to\C$ to \eqref{nls} such that $u$ blows up forward in time, $\inf_{t\in[0,\infty)}N(t)\geq 1$, and $\vert x(t)\vert\lesssim_u \int_0^t N(s)\,ds$ for all $t\geq 0$.
								\end{theorem}	
\begin{proof}
Suppose $u$ were such a solution. In particular $u$ is nonzero, so that by Remark \ref{arzela ascoli} we may find $C(u)>0$ such that
			$$\int_{\vert x-x(t)\vert\leq\frac{C(u)}{N(t)}}\vert u(t,x)\vert^{\frac{2d}{d-1}}\,dx\gtrsim_u 1\quad
			\text{uniformly for }t\in[0,\infty).$$
Applying H\"older and rearranging, this implies
							\begin{equation}
							\label{ls lower bound}
	\int_{\vert x-x(t)\vert\leq\frac{C(u)}{N(t)}}\vert u(t,x)\vert^{\frac{2(d+1)}{d-1}}\,dx\gtrsim_u N(t)
	\quad\text{uniformly for }t\in[0,\infty).
							\end{equation}

We now let $T>1$ and use $u\in L_t^\infty\dot{H}_x^{1/2}$, \eqref{lsi}, and \eqref{ls lower bound} to see
		\begin{align*}
		1\gtrsim_u \int_1^T\int_{\vert x-x(t)\vert\leq\frac{C(u)}{N(t)}}
		\frac{\vert u(t,x)\vert^{\frac{2(d+1)}{d-1}}}{\vert x\vert}\,dx\,dt
		 \gtrsim_u \int_1^T\frac{N(t)}{\vert x(t)\vert+N(t)^{-1}}\,dt.
		\end{align*}
As $\inf_{t\in[1,\infty)} N(t)\geq 1$, to derive a contradiction it will suffice to show that
		\begin{equation}\label{contradiction}
		\lim_{T\to\infty}\int_1^T\frac{N(t)}{1+\vert x(t)\vert}\,dt=\infty.
		\end{equation}		
Recalling that $\vert x(t)\vert\lesssim_u \int_0^t N(s)\,ds$ for all $t\geq 0$, we get
	\begin{align*}
	\int_1^T\frac{N(t)}{1+\vert x(t)\vert}\,dt\gtrsim_u \int_1^T\frac{d}{dt}\log
					\bigg(1+\int_0^t N(s)\,ds\bigg)\,dt
					\gtrsim_u \log\bigg(\frac{1+\int_0^T N(s)\,ds}{1+\int_0^1 N(s)\,ds}\bigg).
	\end{align*}
As $\inf_{t\in[1,\infty)}N(t)\geq 1$, we conclude that \eqref{contradiction} holds, as needed. \end{proof}

\end{document}